\documentclass[reqno]{amsart}

\usepackage{tabu}
\usepackage{amssymb}
\usepackage{mathtools}
\usepackage{a4wide,amsmath}
\usepackage{mathrsfs}
\usepackage{amsthm}
\usepackage{dsfont}
\numberwithin{equation}{section}
\numberwithin{figure}{section}
\numberwithin{table}{section}
\usepackage{bbm}
\usepackage{subfig}
\usepackage{enumerate}
\usepackage[section]{placeins}
\usepackage{graphicx}
\usepackage{ifpdf}
\ifpdf
\DeclareGraphicsExtensions{.pdf,.eps,.jpg,.png}	
\usepackage[suffix=]{epstopdf}
\fi
\usepackage{xcolor}
\usepackage[utf8]{inputenc}
\usepackage{hyperref}
\hypersetup{hidelinks}

\usepackage{listings}
\lstnewenvironment{mat}
{\lstset{language=mathematica,mathescape,columns=flexible}}
{}

\mathtoolsset{showonlyrefs}

\theoremstyle{plain}
\newtheorem{theorem}{Theorem}[section]
\newtheorem{lemma}[theorem]{Lemma}

\newtheorem{corollary}[theorem]{Corollary}
\theoremstyle{definition}

\theoremstyle{remark}

\newcommand{\norm}[1]{\lVert#1\rVert}
\newcommand{\abs}[1]{\lvert#1\rvert} 
\newcommand{\inner}[1]{\langle#1\rangle} 
\newcommand{\essinf}{\mathop{\textup{ess\,inf}}}

\newcommand{\spanm}{\mathop{\textup{span}}}
\newcommand{\di}{\mathrm{d}}   
\newcommand{\R}{\mathbb{R}}
\newcommand{\N}{\mathbb{N}}
\newcommand{\C}{\mathbb{C}}

\newcommand{\sphe}{{\partial B}}

\begin{document}

\title[Linearized problem of electrical impedance tomography]{Eigenstructure of the linearized electrical impedance tomography problem under radial perturbations}

\author[M.~Hirvensalo]{Markus Hirvensalo}
\address[M.~Hirvensalo]{Department of Mathematics and Systems Analysis, Aalto University, P.O. Box~11100, 00076 Helsinki, Finland.}
\email{markus.hirvensalo@aalto.fi}

\begin{abstract}
    We analyze the Fr\'echet derivative \( F \), that maps a perturbation in conductivity to the linearized change in boundary measurements governed by the conductivity equation. 
    The domain is taken to be the unit ball \( B \subset \R^d \) with \( d \geq 2 \), and we choose perturbations \( \eta \) from the Hilbert space \( L^2(B) \).
    Under the condition that the perturbations are rotationally symmetric, we show that the eigenfunctions of the linear approximation \( F \eta \) correspond to the spherical harmonics.
    Furthermore, we establish an explicit formula for the associated eigenvalues and show that for perturbations from any bounded subset, the decay of these eigenvalues is uniform with respect to the degree of the spherical harmonics.
    The established structure of \( F \eta \) enables us to show that the Fr\'echet derivative \( F \) can be approximated by finite-rank operators when restricted to rotationally symmetric perturbations.
    Both the extension to \( L^2(B) \) perturbations and the approximability by finite-rank operators are favorable properties for further analysis of \( F \) in numerical algorithms.
\end{abstract}
%

\maketitle

\section{Introduction} \label{sec:intro}

We study the classical conductivity equation, which describes how an electric potential is induced within a conductive medium by an input current density applied on its surface. 
Our analysis is carried out in a simplified setting, where the domain is the Euclidean unit ball \( B \subset \R^d \) in \( d \geq 2 \) spatial dimensions, with boundary \( \sphe \) (the unit sphere).

The conductivity of the medium is described by a real-valued\footnote{Unless explicitly indicated, the function spaces considered in this work have \( \C \) as the multiplier field.} coefficient \( \gamma \in L^\infty(B; \R) \), which we assume to be strictly positive in the sense that \( \essinf \gamma > 0 \).
Given a surface current density
\begin{equation*}
    f \in L^2_\diamond(\sphe) := \{ g\in L^2(\sphe) \mid \inner{g,1}_{L^2(\sphe)} = 0 \},
\end{equation*}
the conductivity equation states that the corresponding electric potential \( u \) satisfies
\begin{equation} \label{eq:strong_form}
    -\mathrm{div}(\gamma\nabla u) = 0 \text{ in } B, \qquad \gamma \,\frac{\partial u}{\partial r} = f \text{ on } \sphe,
\end{equation}
where \( r \) is the radial coordinate.

The problem \eqref{eq:strong_form} admits a unique weak solution \( u^\gamma_f \in H^1_\diamond(B) \) by virtue of the Lax–Milgram lemma. 
The Sobolev space for the solutions is defined as
\begin{equation*}
    H_\diamond^1(B) := \{w\in H^1(B) \mid w|_{\sphe} \in L^2_\diamond(\sphe)\}.
\end{equation*}
We may define the Neumann-to-Dirichlet (ND) map:
\[ \Lambda(\gamma) : f \to u^\gamma_f |_{\sphe}, \]
which assigns each admissible boundary current \( f \) to the resulting potential observed on the boundary \( \sphe \). 
The ND map is a compact and self-adjoint operator on \(  L^2_\diamond(\sphe) \).

It is well known that the nonlinear forward map \( \gamma\mapsto\Lambda(\gamma) \) is Fr\'echet differentiable with respect to a complex-valued perturbations \( \eta \). 
We denote by \( F \eta := D \Lambda(1; \eta) \) the Fr\'echet derivative of \( \Lambda \) at \( \gamma = 1 \) in the direction of \( \eta \). 
If \( u_f \) and \( u_g \) are harmonic functions in \( B \) with \( f \) and \( g \) as their Neumann boundary values, respectively, then \( F\in \mathscr{L}(L^\infty(B),\mathscr{L}(L^2_\diamond(\sphe))) \) is characterized by
\begin{equation} \label{eq:frechet_redef}
    \inner{(F \eta) f, g}_{L^2(\sphe)} = -\int_B \eta \nabla u_f \cdot \overline{\nabla u_g} \,\di x,
\end{equation}
for all \( \eta \in L^\infty(B) \) and \( f, g \in L^2_\diamond(\sphe) \).

The operator \( F \) arises in certain inverse boundary value problems, such as electrical impedance tomography (EIT).
A practical application is the single-step linearization approximation, where we assume \( \Lambda(1) \) is known or simulated and aim to recover a perturbation \( \eta \) from information on \( \Lambda(1 + \eta) \).
This approach has shown practical success (e.g., \cite{Cheney90}).
For a more comprehensive overview of EIT, we refer the reader to the review articles \cite{Cheney99, Uhlmann2009, Borcea2002a}.

While standard theory guarantees the continuity of \( F : L^\infty(B) \to \mathscr{L}(L^2_\diamond(\sphe)) \), it offers little insight into its approximation by finite-rank operators.
According to \cite[Appendix~A]{Garde2025}, the mapping \( F \) continuously extends to perturbations on \( L^d(B) \), but its approximability by finite-rank operators and compactness remain unstudied.
In this work, we consider the Fr\'echet derivative \( F \) on the subspace \( \mathcal{R} \subset L^2(B) \) consisting of rotationally symmetric perturbations (cf. \eqref{eq:space_of_eta}).
We then study \( F \) in a stronger topology, more specifically as a mapping \( F : \mathcal{R} \to \mathscr{L}(L^2_\diamond(\sphe)) \), and prove that \( F \) can be approximated by finite-rank operators.

To the best of our knowledge, the extension to \( L^2(B) \) perturbations has only been shown in two spatial dimension (see the article by Garde and Hyv\"onen \cite{GH2024}), here we establish the extension in any spatial dimension \( d \geq 2 \) for any \( \eta \) in an infinite dimensional subspace \( \mathcal{R} \subset L^2(B) \).
We believe that this is the first result proving that the Fr\'echet derivative is compact and approximable by finite-rank operators with respect to a topology stronger (or equally strong) than the standard \( \mathscr{L}(L^\infty(B), \mathscr{L}(L^2_\diamond(\sphe))) \) topology.
The extension to \( \mathcal{R} \) perturbations and the approximability by finite-rank operators are desired properties from the standpoint of numerical implementations.

The (linearized) Calder{\'o}n problem in the radial setting has recently been studied in the context of Born approximations \cite{Barcelo2022, Barcelo2024}.
Closely related to this approach, the reader may also find the recent preprints \cite{Daude2026, Castro2026} of interest.

\subsection{Article structure}

In summary, the remainder of this text is structured as follows.
Section~\ref{sec:main_results} presents the main results, including Theorem~\ref{main_thm}, which characterizes the eigenstructure of the linear approximation \( F \eta \) for \( \eta \) belonging to an infinite-dimensional subspace of \( L^2(B) \), and Corollary~\ref{main:corollary}, which establishes the compactness of \( F \) when restricted to the same subspace.
Section~\ref{sec:basis} introduces two auxiliary results, Lemma~\ref{lemma:gradient} and Lemma~\ref{lemma:Jacobi}, which serve as key tools in the proof of the main theorem.
Section~\ref{sec:proof} contains the proof of Theorem~\ref{main_thm} and the auxiliary result Lemma~\ref{proof:lemma_aux}.

\section{The main results} \label{sec:main_results}

The unit ball geometry \( B \) is highly symmetric, and thus it is natural to adopt polar coordinates.
For any \( x \in B \setminus \{ 0 \} \), we write
\begin{equation} \label{eq:polar_coordinates}
    x = r \theta, \quad \text{where} \quad r = \norm{x}, \quad \theta = \frac{x}{\norm{x}}.
\end{equation}
Here, \( r \in (0, 1) \) denotes the radial distance from the origin and \( \theta \in \sphe \) represents the angular component.
Further, we denote the Euclidean norm by \( \norm{x} \).

In polar coordinates, the function space \( L^2(B) \) can be decomposed via separation of variables into a weighted radial space \( L^2_{r^{d-1}}((0,1)) \) and an angular (surface) space \( L^2(\sphe) \).
Here, \( L^2_{r^{d-1}}((0,1)) \) is the space of square-integrable functions on the interval \( (0, 1) \) with respect to the weight \( r^{d-1} \), and \( L^2(\sphe) \) is the standard \( L^2 \) space on the unit sphere.

A natural orthonormal basis for the radial space is given by certain Jacobi polynomials denoted as \( P_k(r) \), where \( k \) indicates the degree of the polynomial.
For the angular component \( L^2(\sphe) \), the spherical harmonics \( f_{\ell, m}(\theta) \), indexed by degree \( \ell \) and order \( m \), provide an orthonormal basis.
Further details regarding these basis functions are provided in Section~\ref{sec:basis}.

Before stating the main result of this work, we define the space of \emph{radial perturbations} as
\begin{equation} \label{eq:space_of_eta}
    \mathcal{R} := \left\{ \eta \in L^2(B) \mid \eta(r, \theta) = \eta(r, \theta') \text{ for almost every } r \in (0, 1) \text{ and } \theta, \theta' \in \sphe \right\}.
\end{equation}
In essence, \( \mathcal{R} \) consists of functions in \( L^2(B) \) that are independent of the angular coordinate \( \theta \). 
That is, for each fixed radius \( r \in (0, 1) \), the function \( \eta(r,\,\cdot\,) \) is constant almost everywhere on the sphere.
In what follows, we abuse the notation by denoting with \( F \) the restriction of \( F \) to \( \mathcal{R} \).

Since any \( \eta \in \mathcal{R} \) is essentially a function depending only on the variable \( r \), it can be expanded solely in the Jacobi polynomials
\begin{equation} \label{eq:radial_expa}
    \eta(r) = \sum^\infty_{k=0} a_k(\eta) \, P_k(r), \quad a_k(\eta) := \int^1_0 \eta(r) \, P_k(r) \, r^{d-1} \di r.
\end{equation}
\begin{theorem} \label{main_thm}
    Let \( \eta \in \mathcal{R} \) be expanded as in \eqref{eq:radial_expa} and let \( f_{\ell,m} \) denote the spherical harmonics on the sphere \( \sphe \) with degree \( \ell \geq 1 \) and order \( m \).
    The operator \( F \eta : L^2_\diamond(\sphe) \to L^2_\diamond(\sphe) \) expanded in the spherical harmonic basis is diagonal and has an eigenstructure of the form
    \begin{equation}
        F \eta = \sum_{\ell, m} \lambda_\ell(\eta) \, \inner{\,\cdot\,, f_{\ell, m}}_{L^2(\sphe)} f_{\ell, m},
    \end{equation}
    where
    \begin{equation*}
        \lambda_\ell(\eta) = \sum^{2 \ell - 2}_{k=0} (-1)^{k+1} \, a_k(\eta) \, \frac{\sqrt{2k+d}}{\ell} \frac{(2\ell-2+d)! \, (2\ell-2)!}{(2\ell-2+d+k)! \, (2\ell-2-k)!}.
    \end{equation*}
    The eigenvalues \( \lambda_\ell(\eta) \) are independent of the index \( m \) and are bounded from above as
    \begin{equation}
        \abs{\lambda_\ell(\eta)} \leq C_d \, \| \eta \|_{L^2(B)} \, \ell^{-\frac{1}{2}},
    \end{equation}
    where the constant \( C_d \) depends only on the dimension \( d \geq 2 \).
\end{theorem}

Interestingly, Theorem~\ref{main_thm} establishes that the linear approximation \( F \eta \) is compact for every \( \eta \in \mathcal{R} \), which is not clear based on the compactness of the original ND map since \( \mathcal{R} \) is not a subspace of \( L^\infty(B) \).
Moreover, if we consider perturbations \( \eta \in \mathcal{R} \) with a fixed norm upper bound \( \| \eta \|_{L^2(B)} \leq \omega \), where \( \omega > 0 \) is some fixed constant, then the eigenvalues \( \lambda_\ell(\eta) \) decay uniformly in \( \ell \).
This suggests that the Fr\'echet derivative \( F \), viewed as an element of  \( \mathscr{L}\big( \mathcal{R}, \mathscr{L}(L^2_\diamond(\sphe)) \big) \) is compact.

A key ingredient in this argument is a general result from Banach space theory: given Banach spaces \( X \) and \( Y \), a linear and bounded operator \( A \in \mathscr{L}(X, Y) \) is compact if it can be approximated in the operator norm by a sequence of finite-rank operators. 
Here, the operator norm is defined in the usual way by
\[ \| A \|_{\rm op} := \sup_{x \in X, \, \| x \|_X=1} \| Ax \|_Y. \]
This result follows from the fact that the space of compact operators \( \mathscr{K}(X, Y) \subset \mathscr{L}(X, Y) \) is closed in the operator topology, and that every finite-rank operator is compact.

These observations yield the following compactness result for the Fr\'echet derivative.
\begin{corollary} \label{main:corollary}
    The Fr\'echet derivative \( F \in \mathscr{L}\left( \mathcal{R}, \mathscr{L}(L^2_\diamond(\sphe)) \right) \) can be approximated by finite-rank operators in the operator topology.
\end{corollary}
\begin{proof}
    Let \( \eta \in \mathcal{R} \) and denote
    \begin{equation*}
        F \eta = \sum^\infty_{\ell=1} \lambda_\ell(\eta) \Psi_{\ell}, \quad \Psi_{\ell} := \sum_{m \in \mathcal{I_\ell}} \inner{\,\cdot\,, f_{\ell, m}}_{L^2(\sphe)} f_{\ell, m}.
    \end{equation*}
    For each \( \ell \in \N \), the orthogonal projection operator \( \Psi_{\ell} \in \mathscr{L}(L^2_\diamond(\sphe)) \) and the linear functional \( \lambda_\ell: \mathcal{R} \to \C \) are fixed.
    Thus, a rank \( N \in \N \) approximation of \( F \) is obtained by truncating its eigendecomposition,
    \[ F_N \eta := \sum_{\ell=1}^N \lambda_\ell(\eta) \Psi_{\ell}. \]

    Both \( F \eta \) and \( F_N \eta \) are diagonal (normal) operators in the considered basis, thereby the singular values of the difference operator are given by the absolute value of the diagonal entries
    \begin{equation*}
        \sigma^N_\ell(\eta) := \abs{\inner{\left( (F - F_N) \eta\right) f_{\ell, m}, f_{\ell, m}}_{L^2(\sphe)}} =
        \begin{cases}
            \abs{\lambda_\ell(\eta)}, &\text{if} \; N < \ell, \\
            0, &\text{if} \; N \geq \ell.
        \end{cases}
    \end{equation*}

    For a compact and normal operator \( A : X \to X \) on a separable Hilbert space \( X \), the operator norm can be given
    \[ \| A \|_{\rm op} = \max_{i \in I} \, \sigma_i, \]
    where \( \{ \sigma_i \}_{i \in I} \) is the set of singular values for \( A \).
    We use this result for \( A = (F - F_N) \eta \) in the following.
    By taking the limit with respect to the truncation index \( N \to \infty \), Theorem~\ref{main_thm} implies
    \begin{equation*}
        \lim_{N \to \infty} \| F - F_N \|_{\mathrm{op}} 
        = 
        \lim_{N \to \infty} \, \max_{\ell, m} \bigg( \sup_{\eta \in \mathcal{R}, \; \| \eta \|_{L^2(B)} = 1} \sigma^N_\ell(\eta) \bigg)
        \leq 
        C_d \, \lim_{N \to \infty} \max_{\ell > N} \, \ell^{-\frac{1}{2}} = 0.
    \end{equation*}
\end{proof}

In an earlier study \cite[Section~6.1]{Autio2025}, the radial singular functions of truncated \( F \) were computed numerically.
At the time, these functions were not explicitly interpreted as approximations for singular functions of \( F \) itself.
However, Corollary~\ref{main:corollary} suggests that repeating the numerical experiment within the framework and assumptions of the present work would yield approximations of the radial singular vectors of \( F \).

\section{On the basis elements} \label{sec:basis}

In this section, some basic properties of the basis elements \( P_k(r) \) and \( f_{\ell,m}(\theta) \) are discussed.
In particular, we prove two intermediate results Lemma~\ref{lemma:gradient} and Lemma~\ref{lemma:Jacobi} that are the main tools in the proof of Theorem~\ref{main_thm}.

\subsection{Spherical harmonics} \label{subsec:sph}

We start by recalling some facts about spherical harmonics; for additional insights on the topic, consult for example \cite[Chapter~5]{Axler_2001} and \cite{Efthimiou_2014}.

Let \( \alpha \in \N^d_0 \) be a \( d \)-tuple and use a multi-index notion \( x^\alpha = \prod_{i=1}^d x_i^{\alpha_i} \) and \( \abs{\alpha} = \sum_{i=1}^d \alpha_i \).
Then, we may define the space of \emph{harmonic homogeneous} polynomials of degree \( \ell \in N_0 \) as
\begin{equation*}
    \mathcal{P}_\ell = \Bigl\{ p : \R^d \to \C  \;\Big|\;  p(x) = \sum_{|\alpha| = \ell} c_\alpha x^\alpha, \; c_\alpha \in \C, \; \Delta p = 0 \Bigr\}.
\end{equation*}
The space of spherical harmonics of degree \( \ell \) is defined as the restriction of the harmonic homogeneous polynomials of degree \( \ell \) onto the unit sphere:
\begin{equation*}
    \mathcal{H}_\ell = \{ p|_{\sphe} \mid p \in \mathcal{P}_\ell \}.
\end{equation*}
We have 
\begin{equation} \label{dimHl}
    \dim\mathcal{H}_\ell = \dim\mathcal{P}_\ell = \binom{\ell+d-1}{d-1} - \binom{\ell+d-3}{d-1},
\end{equation}
with the convention \( \binom{m}{k} = 0 \) for nonnegative integers \( m < k \).

The \( \mathcal{H}_\ell \)-spaces are the eigenspaces for the Laplace–Beltrami operator \( \Delta_{\sphe} \) with
\begin{equation} \label{eq:LBeig}
    \Delta_{\sphe} f = -\ell(\ell+d-2)f, \quad f\in\mathcal{H}_\ell.
\end{equation}
Hence, there are the following orthogonal decompositions:
\begin{equation*}
    L^2(\sphe) = \bigoplus_{\ell=0}^\infty \mathcal{H}_\ell \quad \text{and} \quad L^2_\diamond(\sphe) = \bigoplus_{\ell=1}^\infty \mathcal{H}_\ell.
\end{equation*}
Let \( \{f_{\ell,m}\}_{m \in \mathcal{I}_\ell} \) be an orthonormal basis for \( \mathcal{H}_\ell \); in particular, the index set \( \mathcal{I}_\ell \) has the cardinality \( \dim\mathcal{H}_\ell \).
Thereby, \( \{f_{\ell,m}\}_{\ell\in\N_0, m \in \mathcal{I_{\ell}}} \) is an orthonormal basis for \( L^2(\sphe) \) and similarly \( \{ f_{\ell,m} \}_{\ell \in \N, m \in \mathcal{I_\ell}} \) is an orthonormal basis for \( L^2_\diamond(\sphe) \).
Further, we can show that the orthogonality of the spherical harmonics extends to their gradients, which can be deduced from the following Lemma.
\begin{lemma} \label{lemma:gradient}
    Let \( f_{\ell} \in \mathcal{H}_\ell \) and \( f_{\ell'} \in \mathcal{H}_{\ell'} \), with arbitrary degrees \( \ell, \ell' \in \N_0 \), then
    \begin{equation*}
        \int_{\sphe} \nabla_{\sphe} f_{\ell} \cdot \overline{\nabla_{\sphe} f_{\ell'}} \, \di S
   	    =
        \ell (\ell + d - 2) \int_{\sphe} f_{\ell} \, \overline{f_{\ell'}} \, \di S.
    \end{equation*}
\end{lemma}
\begin{proof}
    In the following, let \( \nabla \) denote the gradient in \( B \) and \( \nabla_\sphe \) denote the surface gradient.
    Let us artificially extend \( f_\ell \) and \( f_{\ell'} \) to \( B \setminus \{ 0 \} \) as constant in the radial dimension.
    In polar coordinates we have
    \begin{equation} \label{help1}
        \big(\nabla f_\ell \cdot \overline{\nabla f_{\ell'}}\big)(r, \theta) = r^{-2} \big(\nabla_\sphe f_\ell \cdot \overline{\nabla_\sphe f_{\ell'}}\big)(1, \theta),
    \end{equation}
    and from \eqref{eq:LBeig} it follows that
    \begin{equation} \label{help2}
        \big(\Delta f_\ell\big)(r, \theta) = r^{-2} \big(\Delta_\sphe f_\ell\big)(1, \theta) = r^{-2} \ell (\ell + d - 2) f_\ell(\theta).
    \end{equation}

    Define \( B_\varepsilon \) to be an origin centered ball with radius \( \varepsilon \in (0, 1) \) and consider the integral
    \begin{equation*}
        \int_{B \setminus B_\varepsilon} \nabla f_\ell \cdot \overline{\nabla f_{\ell'}} \, \di x
        =
        \int_\sphe \overline{f_{\ell'}} \,  \frac{\partial f_\ell}{\partial r} \, \di S - \int_{\sphe_\varepsilon} \overline{f_{\ell'}} \,  \frac{\partial f_\ell}{\partial r} \, \di S - \int_{B \setminus B_\varepsilon} \overline{f_{\ell'}} \,  \Delta f_\ell \, \di x.
    \end{equation*}
    For a radially constant function \( f_\ell \), the normal derivatives vanish and thus the surface integrals are zero.
    From \eqref{help1} and \eqref{help2} we see that the radial contributions cancel from both sides and the claim follows.
\end{proof}

\subsection{Jacobi polynomials}

We define the following \emph{shifted} and \emph{normalized} Jacobi polynomials
\begin{equation} \label{eq:Jacobi}
    P_k(r) = \sqrt{2k + d} \, \sum^k_{q = 0} (-1)^{q} \binom{k}{q} \binom{k + q + d-1}{k} \, r^q, \quad k \in \N_0, \quad r \in (0, 1).
\end{equation}
The definition of \eqref{eq:Jacobi} is in accordance with \cite[\S 18.5(iii)~Eq.~18.5.7]{NIST}, after using a change of variables \( x \mapsto 1 - 2r \), fixing the parameters \( \alpha = d-1 \) and \( \beta = 0 \), and normalizing with respect to the \( L^2_{r^{d-1}((0,1))} \) inner product.
With these modification, we have orthonormality with respect to the order index \( k \),
\begin{equation} \label{eq:Jacobi_inner}
    \inner{P_k(r), P_{k'}(r)}_{L^2_{r^{d-1}}((0,1))} = \int^1_0 P_k(r) \, P_{k'}(r) \, r^{d-1} \di r = \delta_{k,k'}.
\end{equation}
We denote the Kronecker delta by \( \delta_{k,k'} \) in equation \eqref{eq:Jacobi_inner}.
Correspondingly, the recurrence identity \cite[\S 18.9(i)~Eq.~18.9.2\_1]{NIST} takes the form
\begin{multline} \label{eq:Jacobi_recurrence}
    r P_k(r) = -\sqrt{\frac{2k+d}{2k+d-2}} \frac{k (k + d - 1)}{(2k+d-1)(2k+d)} P_{k-1}(r) \\
    + \frac{1}{2} \left( \frac{(d-1)^2}{(2k+d-1)(2k+d+1)} + 1 \right) P_k(r) 
    - \sqrt{\frac{2k+d}{2k+d+2}} \frac{(k+1) (k+d)}{(2k + d)(2k + d + 1)} P_{k+1}(r),
\end{multline} 
with the convention \( P_{-1}(r) = 0 \).

From the orthonormality condition \eqref{eq:Jacobi_inner}, we see that the set of polynomials \eqref{eq:Jacobi} are linearly independent in \( L^2_{r^{d-1}}((0,1)) \).
We can use the Weierstrass approximation theorem to show that monomials of the type \( r^k \) with \( k \in \N_0 \) are dense in \( L^2_{r^{d-1}}((0,1)) \).
Density of \eqref{eq:Jacobi} in \( L^2_{r^{d-1}}((0,1)) \) is inherited from these monomials, since for a fixed \( n \in \N_0 \), we have
\[ \spanm \{ r^k \mid 0 \leq k \leq n \} = \spanm \{ P_k(r) \mid 0 \leq k \leq n \}, \]
which can be directly observed from the following lemma.
\begin{lemma} \label{lemma:Jacobi}
    Let \( k \in \N_0 \), then
    \begin{equation} \label{eq:Monomial_exp}
        r^k = \sum^k_{q=0}  \chi_{k,q} \, P_q(r), \quad \text{where} \quad \chi_{k,q} = (-1)^q \, \frac{\sqrt{2q+d} \, (k + d - 1)! \, k!}{ (k + d + q)! \, (k - q)!}.
    \end{equation}
\end{lemma}
\begin{proof}
    We may use induction to prove the result.
    The base case \( k = 0 \), is easily verified
    \[ r^0 = \sum^0_{q=0} \chi_{0,0} \, P_0(r) = 1. \]

    We initialize the induction step by assuming that \eqref{eq:Monomial_exp} holds for some fixed \( k \geq 0 \).
    Further, let us write the recurrence identity \eqref{eq:Jacobi_recurrence} in a condensed form
    \[ rP_k(r) = A_k P_{k-1}(r) + B_k P_k(r) + C_k P_{k+1}(r), \]
    where the labels \( A_k, B_k \text{ and } C_k \) correspond to the three coefficients in \eqref{eq:Jacobi_recurrence}, respectively.
    Finally, a direct calculation reveals that
    \begin{align*}
        r^{k+1} &= r \, \Big( \sum^k_{q=0} \chi_{k,q} \, P_q(r) \Big) \\[-3mm]
        &= (\chi_{k,1} A_1 + \chi_{k,0} B_0) P_0(r) + \sum^{k-1}_{q=1} (\chi_{k,q+1} A_{q+1} + \chi_{k,q} B_q + \chi_{k,q-1} C_{q-1}) P_q(r) \\
        &\quad+ (\chi_{k,k} B_k + \chi_{k,k-1} C_{k-1} ) P_k(r) + \chi_{k,k} C_k P_{k+1}(r) \\
        &= \chi_{k+1,0} \, P_0(r) + \sum^{k-1}_{q=1} \chi_{k+1,q} \, P_q(r) + \chi_{k+1,k} \, P_k(r) + \chi_{k+1,k+1} \, P_{k+1}(r).
    \end{align*}
    After collecting the terms into a single sum, the proof is completed via a straightforward induction argument.
\end{proof}

\section{Proof of Theorem~\ref{main_thm}} \label{sec:proof}

We begin the proof by assuming that \( \eta \in \mathcal{R} \) is fixed and perform the following series expansion
\begin{equation} \label{proof:eta}
    \eta = \sum^\infty_{k=0} a_k(\eta) \, P_k, \quad a_k(\eta) := \int^1_0 \eta(r) \, P_k(r) \, r^{d-1} \di r,
\end{equation}
where \( P_k \) are the Jacobi polynomials given by \eqref{eq:Jacobi}.

Consider the problem \eqref{eq:strong_form} with \( \gamma = 1 \), and choose a spherical harmonic function \( f_{\ell, m} \in \mathcal{H}_\ell \) as the Neumann boundary value. 
Then, the corresponding harmonic solution is
\begin{equation}
    u_{\ell, m}(r, \theta) = \ell^{-1} r^\ell f_{\ell, m}(\theta), \text{ whenever } \ell > 0.
\end{equation}
We focus on using the basis elements of the zero-mean space \( L^2_\diamond(\sphe) \), and therefore assume \( \ell \geq 1 \).

Inspired by the characterization in \eqref{eq:frechet_redef}, we proceed to calculate
\begin{equation} \label{proof:eq1}
    \inner{(F \eta) f_{\ell, m}, f_{\ell', m'}}_{L^2(\sphe)} = -\int_B \eta \nabla u_{\ell, m} \cdot \overline{\nabla u_{\ell', m'}} \,\di x,
\end{equation}
where
\begin{equation} \label{eq:proof_grad}
    \big( \nabla u_{\ell, m} \cdot \overline{\nabla u_{\ell', m'}} \big) (r, \theta) 
    = r^{\ell+\ell'-2} \big( f_{\ell, m}(\theta) \overline{f_{\ell', m'}(\theta)}
    + (\ell \, \ell')^{-1} \big( \nabla_{\sphe}f_{\ell, m} \cdot \overline{\nabla_{\sphe} f_{\ell', m'}} \big) (1,\theta) \big).
\end{equation}
Combining Lemma~\ref{lemma:gradient} with \eqref{proof:eq1} yields
\begin{equation} \label{proof:innerproduct}
    \inner{(F \eta) f_{\ell, m}, f_{\ell', m'}}_{L^2(\sphe)} = - \frac{2 \ell + d - 2}{\ell} \left( \int^1_0 \eta(r) \, r^{2 \ell - 2} \, r^{d-1} \, \di r \right) \delta_{\ell, \ell'} \delta_{m , m'}.
\end{equation}
We observe that, in order for the inner product to be potentially non-zero, the indices of the spherical harmonics must coincide.
Note that Lemma~\ref{lemma:gradient} is applicable in \eqref{proof:eq1} only because \( \eta \) has essentially no angular component.
Indeed, even a mild angular dependence of \( \eta \) would prevent Lemma~\ref{lemma:gradient} from being applied in its current form.

Expanding \( \eta \) as in \eqref{proof:eta}, applying Lemma~\ref{lemma:Jacobi} to rewrite the monomial \( r^{2 \ell - 2} \) in terms of Jacobi polynomials, and then using their orthonormality, we obtain
\begin{align*}
    \inner{(F \eta) f_{\ell, m}, f_{\ell, m}}_{L^2(\sphe)} &= - \frac{2 \ell + d - 2}{\ell} \sum^{2 \ell - 2}_{k=0} a_k(\eta) \, \chi_{2 \ell - 2, k} \\
    &= \sum^{2 \ell - 2}_{k=0} (-1)^{k+1} a_k(\eta) \frac{\sqrt{2k+d}}{\ell} \frac{(2\ell-2+d)! \, (2\ell-2)!}{(2\ell-2+d+k)! \, (2\ell-2-k)!} \\
    &= \lambda_\ell(\eta).
\end{align*}
It is already clear that the operator \( F \eta \) (defined by \eqref{proof:innerproduct} for \( \eta \in \mathcal{R} \)) expressed in the spherical harmonic basis acts as a diagonal operator
\begin{equation} \label{proof:diagonal}
    (F \eta) : f_{\ell, m} \mapsto \lambda_\ell(\eta) \, f_{\ell, m}.
\end{equation}
However, it remains unclear whether the operator \( F \eta \) is an element of \( \mathscr{L}(L^2_\diamond(\sphe)) \). 
To address this, we proceed by analyzing the behavior of the eigenvalues \( \lambda_\ell(\eta) \).

Let us keep \( d \geq 2 \) fixed, then we may use the Cauchy–Schwarz inequality to show that
\begin{align} \label{proof:proof_up1}
    \abs{\lambda_\ell(\eta)} &\leq \sqrt{\sum^{2 \ell - 2}_{k=0} \abs{a_k(\eta)}^2} \, \sqrt{\sum^{2\ell-2}_{k=0} \frac{2 k+d}{\ell^2} \left( \frac{(2\ell-2+d)! \, (2\ell-2)!}{(2\ell-2+d+k)! \, (2\ell-2-k)!} \right)^2} \notag \\
    &\leq \abs{\sphe}^{-\frac{1}{2}} \| \eta \|_{L^2(B)} \, \sqrt{\sum^{2\ell-2}_{k=0} \frac{2 k+d}{\ell^2} \frac{(2\ell-2+d)! \, (2\ell-2)!}{(2\ell-2+d+k)! \, (2\ell-2-k)!}}.
\end{align}
In the above, we denote the surface area of the unit sphere with \( \abs{\sphe} \). Also, the fraction of factorials is at most \( 1 \), and in order to simplify the analysis we omit squaring it.

Before proceeding with the proof, let us derive an auxiliary result in the form of a lemma.
\begin{lemma} \label{proof:lemma_aux}
    Let \( L \in \N_0 \) and \( d \in \R \), then
    \begin{equation} \label{aux:identity}
        \sum^L_{k=0} (2k+d) \frac{\Gamma(s+1)}{\Gamma(k+s+1)} \frac{L!}{(L-k)!} = L + d.
    \end{equation}
\end{lemma}
\begin{proof}
    Let \( L \in \N_0 \) and \( d \in \R \), and denote \( s := L + d \).
    The case \( L = 0 \) is straightforward, so in what follows we restrict our attention to \( L \geq 1 \), where the argument is nontrivial.

    For \( k=0 \) the summand in \eqref{aux:identity} reduces to \( d \), so we can write
    \begin{equation}
        \mu = d + \sum^L_{k=1} (2k+d) \frac{\Gamma(s+1)}{\Gamma(k+s+1)} \frac{L!}{(L-k)!}.
    \end{equation}
    The goal is to show that \( \mu = d + L \).

    We start by recalling the binomial theorem, which for \( (1+y)^L \) reads
    \begin{equation} \label{aux:binomial_thm}
        \sum^L_{k=0} \binom{L}{k} y^k = (1+y)^L, \quad y \in \R.
    \end{equation}
    We will use \eqref{aux:binomial_thm} to derive two identities used later in this proof.
    The first identity is acquired by differentiating \eqref{aux:binomial_thm} with respect to \( y \), which results in
    \begin{equation} \label{aux:binom_ide1}
        \sum^L_{k=1} k \binom{L}{k} y^{k-1} = L (1+y)^{L-1}.
    \end{equation}
    The second identity is obtained by differentiating \eqref{aux:binomial_thm} twice, multiplying both sides with \( y \) and after that adding \eqref{aux:binom_ide1} on both sides:
    \begin{equation*}
        \left( \sum^L_{k=2} k(k-1) \binom{L}{k} y^{k-2} \right) y + \sum^L_{k=1} k \binom{L}{k} y^{k-1} = \left( L (L-1)(1+y)^{L-2} \right) y + L(1+y)^{L-1},
    \end{equation*}
    which simplifies to
    \begin{equation} \label{aux:binom_ide2}
        \sum^L_{k=1} k^2 \binom{L}{k} y^{k-1} = L (1 + Ly) (1 + y)^{L-2}.
    \end{equation}

    The next step is to use the beta function and write the to-be evaluated sum as an integral.
    The integral form of the beta function for \( z_1,z_2 \in \C \), with positive real parts, i.e., \( \Re(z_1), \Re(z_2) > 0 \) is
    \begin{equation*}
        B(z_1, z_2) = \frac{\Gamma(z_1)\Gamma(z_2)}{\Gamma(z_1+z_2)} = \int^1_0 t^{z_1-1} (1-t)^{z_2-1} \, \di t.
    \end{equation*}
    Thus, for the summation index \( k \geq 1 \) we may write
    \begin{equation} \label{aux:beta_ide}
        \frac{\Gamma(s+1)}{\Gamma(k+s+1)} = \frac{1}{(k-1)!} \int^1_0 t^s (1-t)^{k-1} \, \di t.
    \end{equation}
    Defining
    \begin{equation} \label{aux:sigma}
        \xi(t) := \sum^L_{k=1} (2k+d)k\binom{L}{k}(1-t)^{k-1},
    \end{equation}
    it follows from \eqref{aux:beta_ide} that
    \begin{equation} \label{aux:int}
        \mu = d + \int^1_0 t^s \xi(t) \, \di t.
    \end{equation}

    We now apply the identities \eqref{aux:binom_ide1} and \eqref{aux:binom_ide2} to derive a closed-form expression for \( \xi(t) \). 
    To this end, we substitute \( y = (1-t) \) in \eqref{aux:binom_ide1} and \eqref{aux:binom_ide2}, and form a linear combination of the resulting expressions.
    This yields
    \begin{align}
        \xi(t) &= 2 \left( L \big( 1 + L(1-t) \big) \big(1 + (1-t) \big)^{L-2} \right) + d \left( L \big(1 + (1-t) \big)^{L-1} \right) \notag \\
        &= L(2-t)^{L-2}(2 + 2s -Lt -st).
    \end{align}
    It is straightforward to verify that
    \begin{equation*}
        t^s \xi(t) = \frac{\di}{\di t} \left( L t^{s+1} (2-t)^{L-1} \right).
    \end{equation*}
    Evaluating the integral \eqref{aux:int} using the above identity yields \( \mu = d + L \), completing the proof.
\end{proof}

We now proceed with the proof of Theorem~\ref{main_thm} and apply Lemma~\ref{proof:lemma_aux} to evaluate the sum in \eqref{proof:proof_up1} exactly:
\begin{equation} \label{proof:the_sum}
    \sum^{2\ell-2}_{k=0} \frac{2 k+d}{\ell^2} \frac{(2\ell-2+d)! \, (2\ell-2)!}{(2\ell-2+d+k)! \, (2\ell-2-k)!}
    =
    \frac{2 \ell -2 + d}{\ell^2}.
\end{equation}
Since, \( \ell \geq 1 \) we further estimate:
\begin{equation}
    \frac{2 \ell - 2 + d}{\ell^2} \leq \frac{d+2}{\ell}.
\end{equation}
Finally, we may substitute this estimate into \eqref{proof:proof_up1} and apply the well-known formula
\[ \abs{\sphe} = \frac{2 \pi^{\frac{d}{2}} }{\Gamma \big( \frac{d}{2} \big)} \]
to obtain the desired bound for the eigenvalues,
\begin{equation} \label{proof:lambda_bound}
    \abs{\lambda_\ell(\eta)} \leq C_d \, \| \eta \|_{L^2(B)} \, \ell^{-\frac{1}{2}}, \quad C_d = \sqrt{\frac{(d+2) \, \Gamma \big( \frac{d}{2} \big)}{2 \pi^\frac{d}{2}}}.
\end{equation}

The final step in this proof is to formulate the eigenvalue decomposition of \( F \eta \) as an operator in \( \mathscr{L}(L^2_\diamond(\sphe)) \).
To be quite precise, we have only established that the action of \( F \eta \) on individual spherical harmonics \( f_{\ell, m} \) is well-defined.

Let \( g \in L^2_\diamond(\sphe) \) be arbitrary and let us consider a truncated version of it,
\[ g_N := \sum^N_{\ell = 1} \sum_{m \in \mathcal{I_\ell}} c_{\ell, m} f_{\ell, m}, \quad c_{\ell, m} = \inner{g, f_{\ell, m}}_{L^2(\sphe)}. \]
Here, the cardinality of the index set \( \mathcal{I_\ell} \) is given by \eqref{dimHl}.

The mapping \( (F \eta) g_N \) is still well-defined as it is a \emph{finite} linear combination of elements \( (F \eta) f_{\ell, m} \).
Also, we have \( (F \eta) g_N \in L^2_\diamond(\sphe) \), and thus we may expand it in the spherical harmonic basis, which results in
\begin{equation} \label{proof:FetagN}
    (F \eta) g_N = \sum_{\ell', m'} \Big\langle (F \eta) \Big(\sum^N_{\ell=1} \sum_{m \in \mathcal{I_\ell}} c_{\ell, m} f_{\ell, m} \Big), f_{\ell',m'}  \Big\rangle_{L^2(\sphe)} f_{\ell', m'}.
\end{equation}
From the orthogonality revealed by \eqref{proof:innerproduct}, we see that the inner product in \eqref{proof:FetagN} is potentially non-zero only when the indices of the basis functions match.
Therefore,
\begin{equation} \label{proof:Fetaglim}
    (F \eta) g_N = \sum^N_{\ell = 1} \sum_{m \in \mathcal{I_\ell}} \langle (F \eta) c_{\ell, m} f_{\ell, m}, f_{\ell, m} \rangle_{L^2(\sphe)} f_{\ell, m}
    = \sum^N_{\ell = 1} \sum_{m \in \mathcal{I_\ell}} \lambda_\ell(\eta) \inner{g, f_{\ell, m}}_{L^2(\sphe)} f_{\ell, m}.
\end{equation}
As the truncation index \( N \to \infty \), the right-hand side of \eqref{proof:Fetaglim} converges in \( L^2_\diamond(\sphe) \), due to the fact that the eigenvalues \( \lambda_\ell(\eta) \) accelerate the convergence of a sum that is already convergent.
Hence, the continuous extension of \( F \eta \) is given by the limit
\[ (F \eta) g =  \sum_{\ell, m} \lambda_\ell(\eta) \inner{g, f_{\ell, m}}_{L^2(\sphe)} f_{\ell, m}, \quad g \in L^2_\diamond(\sphe). \]
The proof is now complete.
\hfill \qed

\subsection*{Acknowledgements}

	I wish to thank Nuutti Hyv\"onen for insightful discussions and guidance related to the subject of the paper, and I wish to thank Vanni Noferini for providing a proof for Lemma~\ref{proof:lemma_aux}.
    The project is supported by the Research Council of Finland (decisions 353081 and 359181).
    %

\bibliographystyle{plain}
	

\end{document}